\documentclass[preprint,3p]{elsarticle}
\usepackage{graphicx}
\usepackage{amsmath,amssymb,ifthen}

\usepackage{color}
\definecolor{LinkColour}{rgb}{0,0,0.3} 
\usepackage[colorlinks,citecolor=LinkColour,
linkcolor=LinkColour,urlcolor=LinkColour,
bookmarks=false]{hyperref}

\newtheorem{theorem}{Theorem}[section]
\newtheorem{corollary}[theorem]{Corollary}
\newtheorem{lemma}[theorem]{Lemma}
\newtheorem{identity}[theorem]{Identity}

\newcommand{\per}{\operatorname{per}} 
\newcommand{\abs}[1]{\left\lvert#1\right\rvert}
\providecommand{\Nset}{\mathbb{N}}
\newenvironment{proof}[1][]%
{\noindent\textit{\ifthenelse{\equal{#1}{}}{Proof}{Proof #1}}: }%
{\leavevmode\unskip\nobreak\quad\hspace*{\fill}$\Box$\par}

\begin{document}
\title{Connections between two classes of generalized
 Fibonacci numbers squared and permanents of (0,1) Toeplitz matrices
} \author{Michael A. Allen\corref{cor1}}
\ead{maa5652@gmail.com} \cortext[cor1]{Corresponding author}
\author{Kenneth Edwards} \ead{kenneth.edw@mahidol.ac.th}

\address{Physics Department, Faculty of Science,
Mahidol University, Rama 6 Road, Bangkok
10400, Thailand}

\begin{abstract}
By considering the tiling of an $N$-board (a linear array of $N$
square cells of unit width) with new types of tile that we refer to as
combs, we give a combinatorial interpretation of the product of two
consecutive generalized Fibonacci numbers $s_n$ (where
$s_{n}=\sum_{i=1}^q v_i s_{n-m_i}$, $s_0=1$, $s_{n<0}=0$, where $v_i$
and $m_i$ are positive integers and $m_1<\cdots<m_q$) each raised to
an arbitrary non-negative integer power. A $(w,g;m)$-comb is a tile
composed of $m$ rectangular sub-tiles of dimensions $w\times1$
separated by gaps of width $g$.  The interpretation is used to give
combinatorial proof of new convolution-type identities relating
$s_n^2$ for the cases $q=2$, $v_i=1$, $m_1=M$, $m_2=m+1$ for $M=0,m$
to the permanent of a (0,1) Toeplitz matrix with 3 nonzero diagonals
which are $-2$, $M-1$, and $m$ above the leading diagonal.  When $m=1$
these identities reduce to ones connecting the Padovan and Narayana's
cows numbers.
\end{abstract}


\maketitle

\section{Introduction}
The tiling of an $n$-board (an $n\times1$ board divided into $n$
square cells) using fence tiles has recently been used to obtain quick
combinatorial proofs of various identities, some of which were new,
concerning powers of the Fibonacci numbers \cite{EA19,EA20,EA20a}. A
$(w,g)$-fence is a tile composed of two $w\times1$ sub-tiles separated
by a $g\times1$ gap. The tiling of an $n$-board using tiles with gaps
or tiles of non-integer length can always be expressed in terms of a
tiling with metatiles. A \textit{metatile} is a grouping of tiles that
exactly covers an integer number of cells and cannot be split into
smaller metatiles~\cite{Edw08}. Obtaining an expression for the number
of metatiles of a given length is the key to obtaining some of the
identities which are of a convolution type when the number of possible
metatiles is infinite.

In \cite{EA15} we showed that there is a bijection between strongly
restricted permutations $\pi$ of the set
$\Nset_n=\{1,2,\ldots,n\}$ for which the permissible values of
$\pi(i)-i$ for each $i\in\Nset_n$ are the elements of a finite set $W$
which is independent of $n$~\cite{Leh70} and the tilings of an
$n$-board with a finite number of types of $(\frac12,g)$-fences where
$g\in\{0,1,\ldots\}$. The types of fence and the rules for where they
can be placed are as follows. Each negative element $-g$ of $W$
corresponds to a $(\frac12,g-1)$-fence (denoted by $\bar{F}_{g-1}$)
that must be placed so that the right side of the left sub-tile is
aligned with a cell boundary of the $n$-board. Each non-negative
element $g$ of $W$ corresponds to a $(\frac12,g)$-fence (denoted by
$F_{g}$) that must be placed so that the left side of the left
sub-tile is aligned with a cell boundary.  We denote the number of
strongly restricted permutations of $\Nset_n$ such that $\pi(i)-i\in
W$ by $P_n^W$. 

The permanent of an $n\times n$ matrix $A$ is given by 
\[
\per A =\sum_{i_1,i_2,\ldots,i_n}
\abs{\varepsilon_{i_1i_2\ldots i_n}}A_{1i_1}A_{2i_2}\ldots A_{ni_n}
\]
where $\varepsilon_{i_1i_2\ldots i_n}$ is the permutation symbol. The
permanent is thus like the determinant but with all plus signs in the
sum.  An equivalent definition of $P_n^W$ is the permanent of an
$n\times n$ matrix whose $(i,j)$th entry is 1 if $j-i\in W$ and 0
otherwise \cite{Leh70}. The matrix is a (0,1) Toeplitz matrix such that the
diagonal $w$ above the leading diagonal is nonzero iff $w\in W$. There
is an explicit formula for $P^{\{-k,0,ki\}}$ for $k,i\in1,2,3,\ldots$
\cite{CCR97}. For the general case, there is a five-step procedure for
obtaining the generating function for $P_n^W$ \cite{Bal10}. When $-1$
is the only negative element of $W$, an expression for the recursion
relation for $P_n^W$ can be obtained directly in terms of the elements
of $W$; for convenience we state the following theorem which is a
re-expression of part of Theorem~4.1 in \cite{EA15}.
\begin{theorem}\label{T:-1}
If  $W=\{-1,d_1,\ldots,d_r\}$ where $0\le d_1<d_2<\cdots<d_r$ and
$d_r>0$ then for all $n\ge0$,
\[
P_n^W=P^W_{n-d_1-1}+\cdots+P^W_{n-d_r-1}+\delta_{n,0},
\]
where we take $P^W_{n<0}=0$ and $\delta_{i,j}$ is 1 if $i=j$ and 0 otherwise.
\end{theorem}
We will also use the result 
that $P_n^{W^{(-1)}}=P_n^W$ where $W^{(-1)}=\{x\mid -x\in W\}$.

In Section~\ref{s:C} we generalize the notion of a fence to a tile,
that we refer to as a comb, composed of an arbitrary number of
sub-tiles. We show that tiling an $n$-board with certain types of
combs leads to a combinatorial interpretation of products of two
consecutive generalized Fibonacci numbers each raised to a
non-negative integer power. As with fences, this result can be used to
obtain identities in a quick and intuitive manner once we have found
an expression for the number of metatiles when tiling with the
specified types of comb. As our first example of this, in Sections \ref{s:I}
and \ref{s:II} we obtain simple relations between the numbers of
metatiles of a given length when tiling with two types of combs (with
$M$ or $m+1$ gaps and all sub-tiles and gaps of width $\frac12$) and
$P_n^{\{-2,M-1,m\}}$ for $M=0,m$. This enables us to obtain various
general identities relating the squares of two one-parameter families
of generalized Fibonacci numbers to permanents.  The results for
the $m=1$ cases are shown to reduce to identities connecting the
Padovan and Narayana's cows sequences in Section~\ref{s:cows}.

\section{Combs}\label{s:C}

We define a $(w,g;m)$-\textit{comb} as a linear array of $m$ sub-tiles (which
we refer to as \textit{teeth}) of dimensions $w\times1$ separated by
$m-1$ gaps of width $g$. Evidently, a $(w,g;2)$-comb is a
$(w,g)$-fence, a $(w,g;1)$-comb is a $w\times1$ tile, and a
$(w,0;m)$-comb is a $mw\times1$ tile. The following theorem is a
generalization of Theorem~5 in \cite{EA21}.

\begin{theorem}\label{T:1combs}
If $A_n$ is the number of ways to tile an $n$-board using $v_i$
colours of $(1,p-1;m_i)$-combs for $i=1,\ldots,q$, where $p$ and the
$m_i$ are positive integers and $m_1<\cdots<m_q$, then for $n\ge0$,
\[
A_{pn+r}=s_n^{p-r}s_{n+1}^r, \quad r=0,\ldots,p-1,
\]
where $s_n=v_1s_{n-m_1}+\cdots+v_qs_{n-m_q}+\delta_{n,0}$,
$s_{n<0}=0$.
\end{theorem}
\begin{proof}
We identify the following bijection between the tilings of a
$(pn+r)$-board using $v_i$ colours of $(1,p-1;m_i)$-combs and the
tilings of an ordered $p$-tuple of $r$ $(n+1)$-boards followed by
$p-r$ $n$-boards using $v_i$ colours of $m_i$-ominoes. For
convenience we number the boards in this $p$-tuple from 0 to $p-1$ and
the cells in the $(pn+r)$-board from 0 to $pn+r-1$.  Tile board $j$ in
the $p$-tuple with the contents (taken in order) of the cells of the
given $(pn+r)$-board comb tiling whose cell number modulo $p$ is
$j$. The teeth of any $(1,p-1;m_i)$-comb (which will always lie on
consecutive cells with the same cell number modulo $p$) get mapped to
the same colour of $m_i$-omino in board $j$. The procedure is
reversed by splicing the $m_i$-omino tilings of the $p$-tuple of
boards, hence establishing the bijection. The number of
$m_i$-omino tilings of an $n$-board is $s_n$ \cite{BQ=03}. Hence
the number of $m_i$-omino tilings of the $p$-tuple of boards is
$s_{n+1}^rs_n^{p-r}$ and the result follows.
\end{proof}

\begin{corollary}\label{C:bij}
If $A_n$ is the number of ways to tile an $n$-board using $v_i$
colours of
$(1/p,1-1/p;m_i)$-combs where $p$ and the $m_i$ are positive integers and 
$m_1<\cdots<m_q$, then for $n\ge0$,
\[
A_{n}=s_n^p,
\]
where $s_n=v_1s_{n-m_1}+\cdots+v_qs_{n-m_q}+\delta_{n,0}$, $s_{n<0}=0$.
\end{corollary}
\begin{proof}
Reduce the board and tiles in Theorem~\ref{T:1combs} by a factor of
$p$ lengthwise and discard the non-integer length boards by only
considering the $r=0$ case.
\end{proof}

A \textit{mixed metatile} is a metatile that contains more than one
type of tile \cite{EA20a}. For example, the simplest (and shortest)
mixed metatiles when tiling with half-squares ($h$, which can also be
regarded as $(\frac12,\frac12;1)$-combs) and
$(\frac12,\frac12;m+2)$-combs ($C$) are a $C$ with all the gaps filled
with half-squares (which we refer to as a filled comb) followed by an
$h$ (in order to give a grouping of tiles of integer length) or an $h$
followed by a filled comb (Fig.~\ref{f:hcCmetatiles}). The symbolic
representations of these metatiles are $Ch^{m+2}$ and $hCh^{m+1}$,
respectively.  The fact that there is a pair of simplest mixed
metatiles is a consequence of Lemma~\ref{L:mtpairs}.

\begin{figure}
\begin{center}
\includegraphics[width=12.7cm]{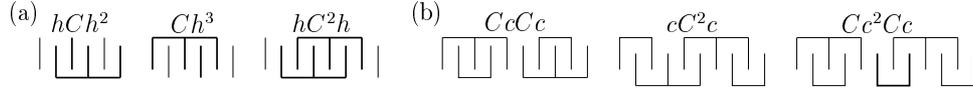}
\end{center}
\caption{Examples of mixed metatiles and their symbolic
  representations when tiling with (a)~half-squares
  ($h$) and $(\frac12,\frac12;m+2)$-combs ($C$)
  (b)~$(\frac12,\frac12;m+1)$-combs ($c$) and $C$ when
  $m=1$. Each upper (lower) vertical line represents a tooth filling a
  left (right) slot. The horizontal lines show which
  teeth are part of the same comb.  Interior tiles are indicated by
  thicker lines. The first two metatiles in each case are examples of
  pairs of metatiles in the sense of Lemma~\ref{L:mtpairs}.}
\label{f:hcCmetatiles}
\end{figure}

We let $\mu_l^{(m_1,m_2)}$ denote the number of mixed
metatiles of length $l$ when tiling with $(\frac12,\frac12;m_1)$- and
$(\frac12,\frac12;m_2)$-combs. When tiling an $n$-board with such
combs or $(\frac12,g)$-fences we refer to the halves of each cell as
\textit{slots}.

\begin{lemma}\label{L:mtpairs}
For $m_2>m_1\ge1$,
the mixed metatiles when tiling with $(\frac12,\frac12;m_1)$- and
$(\frac12,\frac12;m_2)$-combs occur in pairs whereby one element of the
pair is generated from the other by interchanging the contents of the
slots in each cell.
\end{lemma}

\begin{proof}
Interchanging the slot contents corresponds to interchanging the
boards in the bijection described in the proof of
Theorem~\ref{T:1combs}. The operation of interchanging slot contents
never changes a metatile into an arrangement of tiles which is not a
metatile because the operation does not change which cells any given
comb straddles.  The operation will result in a distinct metatile
being produced since all mixed metatiles contain at least one cell
that contains teeth from both types of comb.
\end{proof}

\section{Tiling with half-squares and $(\frac12,\frac12;m+2)$-combs
  where $m\ge0$}\label{s:I}

When tiling with half-squares ($h$) and $(\frac12,\frac12;m+2)$-combs
($C$), the simplest metatiles are $h^2$ and $C^2$ (two interlocking
combs that we will refer to as a \textit{bicomb}) and are of length
$1$ and $m+2$, respectively. All other metatiles are mixed.  

\begin{lemma}
When tiling with $h$ and $C$, the symbolic representation of one
member of each pair of mixed metatiles starts and ends with $h$. 
\end{lemma}
\begin{proof}
The first and last cells of a mixed metatile must contain an $h$ since
otherwise it would be a bicomb. From Lemma~\ref{L:mtpairs}, one member
of the pair starts with $h$. If the $h$ in the final cell is not in
the right-hand slot, it will be inside a gap in a comb and so the
corresponding $h$ in the symbolic representation will still be at the
end.
\end{proof}

We refer to all tiles in a mixed metatile other than the initial and
final $h$ as \textit{interior tiles}. The total length of the
sub-tiles of the interior tiles of a mixed metatile of length $l$
is therefore $l-1$.

A systematic way to generate the symbolic representation of all
possible metatiles is via a directed pseudograph (henceforth referred
to as a digraph) in which each arc represents the addition of a tile
or tiles starting at the next available slot and each node (apart from
the starting 0~node which represents the initial empty board or completed
metatile) corresponds to a particular configuration of the partially
occupied slots starting at the first empty slot of the incomplete
metatile \cite{EA15}. Each metatile corresponds to a walk starting and
finishing at the 0~node without passing through it in between.  The
nodes are identified by binary strings.  A 0 (a 1) in the string
represents an empty (a filled) slot and a string starting with
$\bar{0}$ indicates that the first empty slot is a right-hand side one.

We obtain all configurations of interior tiles by modifying the
digraph approach for finding metatiles. There can be no 0~node in the
digraph since it represents the beginning or end of the whole
metatile. To generate the symbolic representation of the interior
tiles of the member of each possible pair of mixed metatiles starting
with an $h$ we instead start at the $\bar{0}$~node since an $h$ has
already been placed in the first slot of the board. We can end either
at the $\bar{0}$~node which means the final $h$ needed to complete the
mixed metatile will be placed in the right slot, or at the 01~node
whereby the final $h$ will lie in the last gap of the final comb in
the mixed metatile. We refer to these two nodes as \textit{exit
  nodes}.  We would, incidentally, generate the other members of the
pairs of mixed metatiles by instead starting at the 01~node which in
this case would correspond to an $h$ filling the right slot of the
first cell on the board.

The digraphs for generating interior tiles when tiling with $h$ and
$C$ are given in Fig.~\ref{f:hm+1-digraphs}. We use $(01)^2$ to mean
0101, etc., and $(x)^0$ for any $x$ means the empty string.  With this
binary string notation we can represent a $(\frac12,\frac12;p+1)$-comb
by $1(01)^p$. It is then easily seen that an arc representing such a
comb leaving a $(01)^q$~node ($(\bar{0}1)^q$~node) will end at a
$(\bar{0}1)^{\abs{p-q}}\bar{0}$~node ($(01)^{\abs{p-q}}$~node)
where we omit the final $\bar{0}$ when $p\neq q$.  Notice
that there is no $h$ arc leaving either exit node
since this would result in a completed metatile and the corresponding
$h$ would therefore not be an interior tile.

The nodes in the digraph for generating interior tiles can be grouped
into the pairs $(\bar{0}1)^p\bar{0}$ and $(01)^{p+1}$ for $p=0,\ldots,m$
where we omit the final $\bar{0}$ when $p\neq0$. The members of each
pair are equivalent in the sense that the arc from node $A$ to node
$B$ is of the same type as the arc from node $\tilde{A}$ to node
$\tilde{B}$ where $X$ and $\tilde{X}$ are a pair of nodes. This is
apparent from the fact that rotating the digraphs as depicted in
Fig.~\ref{f:hm+1-digraphs} by 180$^\circ$ maps each node into the
other member of its pair and the types and directions of the arcs
remain unchanged by the rotation.

\begin{figure}
\begin{center}
\includegraphics[width=12.7cm]{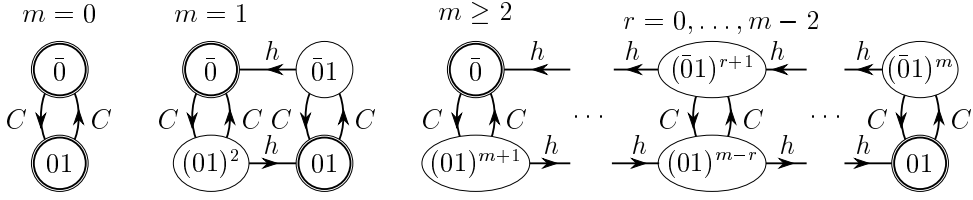}
\end{center}
\caption{Digraphs for generating configurations of interior tiles when
  tiling with half-squares ($h$) and
  $(\frac12,\frac12;m+2)$-combs ($C$).}
\label{f:hm+1-digraphs}
\end{figure}

To show the bijection in the following lemma, in
Fig.~\ref{f:FfS-digraphs} we present the digraphs for obtaining the
symbolic representations of the tilings of an $n$-board using the fences
$F_m$ ($F$), $\bar{F}_1$ ($\bar{f}$), and $\bar{F}_0$ which is just a
square that straddles a cell boundary ($\bar{S}$). Again using our
binary string notation to label nodes, by $\bar{0}^21$ we mean
$\bar{0}01$, etc. The tiles $F$, $\bar{f}$, $\bar{S}$ can be viewed as
the strings $10^{2m}1$, $\bar{1}001$, and $\bar{1}1$, respectively,
where the absence (presence) of a bar on the first digit means that
the left sub-tile of the fence must be placed in the left (right) slot
of a cell.  Since the only tile that can start in a left slot is $F$,
this is the only arc leaving the 0 node. From a $\bar{0}^{2p}1$ node
for $p=1,\ldots,m$ there are two options: adding $\bar{S}$
($\bar{f}F\bar{f}^{p-1}$) leaves $\bar{0}^{2(p-1)}$
($\bar{0}^{2(m-p)}1$). To understand the latter case, note that adding
$\bar{f}$ results in an empty slot that can only be filled by the left
sub-tile of an $F$. This addition of $\bar{f}F$ then leaves
$\bar{0}^{2(m-1)}$ if $p=1$ and $\bar{0}10^{2(p-2)}10^{2(m-p)+1}1$
otherwise.  In the $p>1$ case, it is easily seen that the first
$2(p-1)$ empty slots must be filled by $p-1$ $\bar{f}$ which leaves
$\bar{0}^{2(m-p)}1$.

A tiling of an $n$-board with $(\frac12,g)$-fences
corresponds to any walk that begins and ends at the 0~node and
contains a total of $n$ tiles since the total contribution to the
length from the two sub-tiles making up any $(\frac12,g)$-fence tile
is 1.

\begin{figure}
\begin{center}
\includegraphics[width=12.7cm]{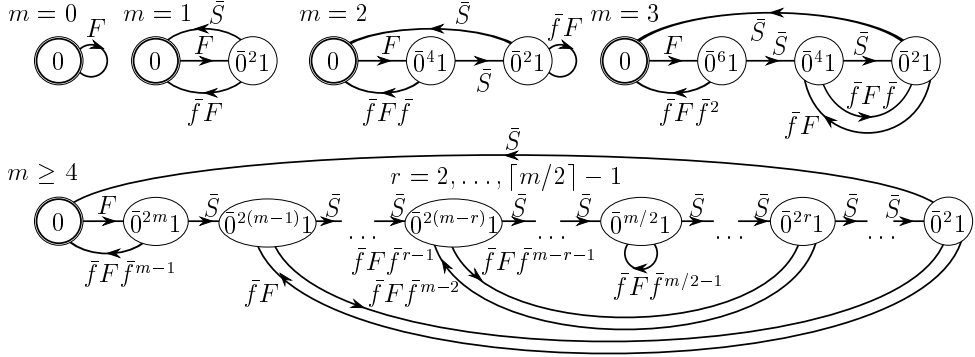}
\end{center}
\caption{Digraphs for tiling using $F_m$ ($F$),
  $\bar{F}_1$ ($\bar{f}$), and $\bar{F}_0$ ($\bar{S}$). The $\bar{0}^{m/2}1$
node is only present if $m$ is even.}
\label{f:FfS-digraphs}
\end{figure}

\begin{lemma}\label{L:bijI}
There is a bijection between the possible pairs of configurations of
interior tiles of total length $n$ (not counting the gap when leaving
the digraph at the 01~node) when tiling
with half-squares and $(\frac12,\frac12;m+2)$-combs and the ways to
tile an $n$-board using $F$, $\bar{f}$, and $\bar{S}$.
\end{lemma}
\begin{proof}
For any walk that starts and finishes at the 0~node there is exactly
one walk that starts at the $\bar{0}$ node and ends at an exit node
such that each $\bar{S}$ arc (arc containing an $F$) corresponds to an
$h$ ($C$) arc in the interior tile digraph. This can be seen from the
fact the $\bar{0}^{2p}1$ node (for $p=0,\ldots,m$ with the $p=0$ case
understood to mean the 0 node) in the fence digraph corresponds to the
$(01)^{p+1}$ pair of nodes in the interior tile digraph in the sense
that if there is an $\bar{S}$ arc (arc containing an $F$) from node
$A$ to node $B$ then there is an $h$ ($C$) arc from $\tilde{A}$ to
$\tilde{B}$, where $\tilde{X}$ denotes a member of the pair of 
interior tile digraph nodes
corresponding to the fence digraph node $X$.  We now show that the
corresponding tilings are of the same length. The total
length of the sub-tiles in a $C$ is $(m+2)/2$. For $m>1$, $C^2$
corresponds to $F$ followed by $\bar{f}F\bar{f}^{m-1}$ which are both
tilings of length $m+2$. For all $m$, $Ch^m$ (a comb with all but the
last gap filled) corresponds to $F\bar{S}^m$ and these tilings are of
length $m+1$, not including the unfilled gap. All metatiles containing
at least one $\bar{S}$ can be obtained by modifying $F\bar{S}^m$ in a
series of steps which each involve adding $\bar{f}F\bar{f}^{m-p-1}$
for some $p$ and adding or removing an appropriate number of
$\bar{S}$. If the $\bar{0}^{2(m-p)}1$ node for any $p=1,\ldots,\lceil
m/2\rceil-1$ is left via the $\bar{f}F\bar{f}^{m-p-1}$ arc then the
number of $\bar{S}$ bypassed is $m-2p$ and so the net increase in the
number of tiles is $m-p+1-(m-2p)=p+1$. For $p=\lfloor
m/2\rfloor+1,\ldots,m-1$, after leaving the $\bar{0}^{2(m-p)}1$ node
we need to add $2p-m$ $\bar{S}$ to return to it. The total number of
tiles thus increases by $m-p+1+2p-m=p+1$. Finally, for even $m$, if
$p=m/2$, executing the $\bar{0}^{m/2}1$ node loop once also gives
$p+1$ extra tiles. The $\bar{0}^{2(m-p)}1$ node corresponds to the
$(01)^{m-p+1}$ pair of nodes in the comb tiling. Leaving either node
of the pair via a $C$ arc results in bypassing $m-p$ $h$ and gaining
$p$ $h$ before an exit node is reached. The total length of tiles
added is thus $\frac12m+1+\frac12(p-(m-p))=p+1$.
\end{proof}

\begin{theorem}\label{T:mu1}
For $m\ge0$ and $l>1$,
\begin{equation}\label{e:mu1}
\mu_l^{(1,m+2)}=2P_{l-1}^{\{-2,-1,m\}}.
\end{equation}
\end{theorem}
\begin{proof}
As there is a bijection between permutations $\pi$ of $\Nset_n$
satisfying $\pi(i)-i\in\{-2,-1,m\}$ and tilings of an $n$-board with
$\bar{f}$, $\bar{S}$, and $F$, it follows that $P_{l-1}^{\{-2,-1,m\}}$
is the number of ways to tile an $(l-1)$-board with $F$, $\bar{f}$,
and $\bar{S}$. By Lemma~\ref{L:bijI} and the fact that the total
length of the interior tiles is 1 less than the complete metatile,
$P_{l-1}^{\{-2,-1,m\}}$ equals the number of pairs of mixed metatiles
of length $l$.
\end{proof}

We let $s_n^{(p,q)}$ denote the $n$th term in the sequence given by
$s_n^{(p,q)}=\delta_{0,n}+s_{n-p}^{(p,q)}+s_{n-q}^{(p,q)}$,
$s_{n<0}^{(p,q)}=0$. However, for ease of reading, in the rest
of this section we write $s_n$ and $\mu_l$ instead of $s_n^{(1,m+2)}$
and $\mu_l^{(1,m+2)}$, respectively.  The proofs of the following
identities are similar to those given in \cite{EA20a} which are in turn
based on techniques expounded in \cite{BQ=03}.

\begin{identity}\label{I:gen1}
For $n\ge0$,
\begin{equation}\label{e:gen1}
s_n^2=\delta_{n,0}+s_{n-1}^2
+s_{n-m-2}^2+2\!\!\!\sum_{l=m+2}^n\!\!\!P_{l-1}^{\{-2,-1,m\}}s_{n-l}^2.
\end{equation}
\end{identity}
\begin{proof}
As in \cite{BHS03,EA15,EA19}, we condition on the 
metatile at the end of the board. If it is of length $l$ there will be
$A_{n-l}$ ways to tile the remaining $n-l$ cells.  There is one
metatile of length 1 ($h^2$), the $C^2$ metatile is of length $m+2$,
and there are $\mu_l$ mixed metatiles
of length $l$ for each $l\geq m+2$.  If $n=l$ there is exactly one tiling
corresponding to that final metatile so we make $A_0=1$. There is no
way to tile an $n$-board if $n<l$ and so $A_{n<0}=0$. Thus
\[
A_{n}=\delta_{n,0}+A_{n-1}+A_{n-m-2}+\sum_{l=m+2}^n\mu_lA_{n-l}.
\]
Applying Corollary~\ref{C:bij} and Theorem~\ref{T:mu1} to this
gives \eqref{e:gen1}.
\end{proof}

\begin{identity}\label{I:n+m+2}
For $n\ge0$,
\begin{equation}\label{e:n+m+2}
s_{n+m+2}^2-1=\sum_{k=0}^n\left\{s_{k}^2+2\sum_{i=0}^kP_{k+m+1-i}^{\{-2,-1,m\}}s_{i}^2\right\}.
\end{equation}
\end{identity}
\begin{proof}
How many ways are there to tile an $(n+m+2)$-board using at least 1
comb? \textit{Answer~1}: $A_{n+m+2}-1$ since this corresponds to all
tilings except the all-$h$ tiling. \textit{Answer~2}: condition on the
location of the last comb. Suppose this comb starts on cell $k+1$ and
therefore ends on cell $k+m+2$ ($k=0,\ldots,n$). Either it is part of
a bicomb which covers cells $k+1$ to $k+m+2$ and so there are $A_k$
ways to tile the remaining cells, or the cells are at the end of a
mixed metatile and so there are
$\mu_{m+2}A_{k}+\mu_{m+3}A_{k-1}+\cdots+\mu_{k+m+2}A_0$ ways to tile
the remaining cells. Hence, equating the two answers,
\[
A_{n+m+2}-1=\sum_{k=0}^n\left\{A_k+\mu_{k+m+2}A_0+\mu_{k+m+1}A_1+\cdots+\mu_{m+3}A_{k-1}+\mu_{m+2}A_k\right\}.
\]
The identity then follows from \eqref{e:mu1} and Corollary~\ref{C:bij}.
\end{proof}

\begin{identity}\label{I:n(m+2)+jI}
For $n\ge0$ and $j=0,\ldots,m+1$,
\begin{equation}\label{e:n(m+2)+jI}
s_{n(m+2)+j}^2=\delta_{j,0}+(1-\delta_{j,0})s_{j-1}^2+\sum_{k=1}^n\left\{s_{k(m+2)+j-1}^2
+2\!\!\!\!\!\!\!\!\sum_{i=0}^{(k-1)(m+2)+j}\!\!\!\!\!\!\!\!P_{k(m+2)+j-1-i}^{\{-2,-1,m\}}s_i^2\right\}.
\end{equation}
\end{identity}
\begin{proof}
How many ways are there to tile an $(n(m+2)+j)$-board using at least
one $h$?  \textit{Answer~1}: $A_{n(m+2)+j}-\delta_{0,j}$ since the
all-bicomb tiling only occurs when $j=0$. \textit{Answer~2}: condition
on the location of the final $h$. This must lie in a cell whose number
modulo $m+2$ equals $j$ since the cells after this, if any, must be
filled with bicombs. Suppose that this $h$ is in cell $k(m+2)+j$
($k=\delta_{j,0},\ldots,n$). Either it is part of $h^2$ and so there
are $A_{k(m+2)+j-1}$ ways to tile the remaining cells, or it is part
of a mixed metatile in which case there are
$\mu_{m+2}A_{k(m+2)+j-m-2}+\mu_{m+3}A_{k(m+2)+j-m-3}+\cdots+\mu_{k(m+2)+j}A_0$
ways to tile the remaining cells. In the latter case, evidently, $k$
cannot be zero. Hence, equating the answers,
\begin{multline*} 
A_{n(m+2)+j}-\delta_{j,0}=\sum_{k=\delta_{j,0}}^nA_{k(m+2)+j-1}\\
\mbox{}+\sum_{k=1}^n\big(\mu_{k(m+2)+j}A_0+\mu_{k(m+2)+j-1}A_1+\cdots+\mu_{m+2}A_{(k-1)(m+2)+j}\big).
\end{multline*} 
Then \eqref{e:n(m+2)+jI} follows
from \eqref{e:mu1} and Corollary~\ref{C:bij}.
\end{proof}

In the following identity we use the fact that the number of ways to
tile an $n$-board using only $h^2$ and $C^2$ is $s_n$ since
this is equivalent to tiling an $n$-board with squares and
$(m+2)$-ominoes \cite{BCCG96,BQ=03}.

\begin{identity}\label{I:s2-sI}
For $n\ge0$,
\begin{equation}\label{e:s2-sI}
s_{n}^2=s_{n}+2\sum_{k=0}^{n-m-2}\sum_{r=m+2}^{n-k}
P_{r-1}^{\{-2,-1,m\}}s_{k}s_{n-k-r}^2.
\end{equation}
\end{identity}
\begin{proof}
How many ways are there to tile an $n$-board using at least 1 mixed
metatile? \textit{Answer~1}: $A_n-s_n$ since $s_n$
is the number of ways to tile an $n$-board without using mixed
metatiles.  \textit{Answer~2}: condition on the position of the first
mixed metatile. If it lies on cells $k+1$ to $k+r$ where
$k=0,\ldots,n-r$ and $r=m+2,\ldots,n-k$, there are
$s_k\mu_rA_{n-k-r}$ ways to tile the board.
Summing over all possible $k$ and $r$ and equating to Answer~1 gives
\[
A_n-s_n=\!\!\!\sum_{\substack{k\ge0,\,r\ge m+2,\\k+r\le n}}\!\!\!
s_k\mu_rA_{n-k-r}.
\]
After re-expressing the right-hand side as a double sum,
the identity follows from \eqref{e:mu1} and Corollary~\ref{C:bij}.
\end{proof}

Note that when $m=0$, since $P_n^{\{-2,-1,0\}}=1$ and
$s_n^{(1,2)}=F_{n+1}$ where $F_n$ is the $n$th Fibonacci number,
Identities \ref{I:gen1}, \ref{I:n+m+2}, \ref{I:n(m+2)+jI} with $j=1$,
and \ref{I:s2-sI} reduce, respectively, to Identities 4.1, 4.2, and
4.3 in \cite{EA19} and Identity~2.10 in \cite{EA20}.

\section{Tiling with $(\frac12,\frac12;m+1)$-  and
  $(\frac12,\frac12;m+2)$-combs where $m\ge1$}\label{s:II}

We use $c$ and $C$ to denote $(\frac12,\frac12;m+1)$- and
$(\frac12,\frac12;m+2)$-combs, respectively. The simplest metatiles
are $c^2$ and $C^2$ which have lengths $m+1$ and $m+2$,
respectively. All other metatiles are mixed.

\begin{lemma}
When tiling with $c$ and $C$, the symbolic representation of any mixed
metatiles begins with $Cc$ or $cC$ and ends with $Cc$.
\end{lemma}
\begin{proof}
A metatile cannot start or end with a $C^2$ or $c^2$ as these are
themselves metatiles. The penultimate comb is always a $C$ whether or
not the tooth right at the end of the tile belongs to the final $c$
in the tiling (Fig.~\ref{f:hcCmetatiles}(b)).
\end{proof}

\begin{corollary}
The smallest mixed metatiles when tiling with $c$ and $C$ are $CcCc$
and $cC^2c$ and are of length $2m+3$.
\end{corollary}

We refer to a comb which is not one of two combs at the beginning or
the two combs at the end of the metatile as an \textit{interior comb}.

\begin{figure}[b]
\begin{center}
\includegraphics[width=12.7cm]{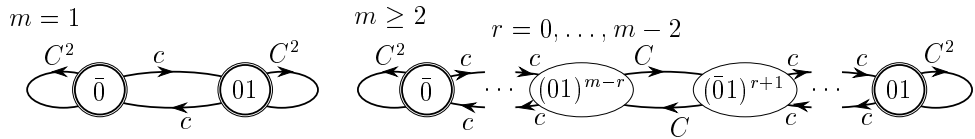}
\end{center}
\caption{Digraphs for generating configurations of interior combs when
  tiling with $(\frac12,\frac12;m+1)$-combs ($c$) and
  $(\frac12,\frac12;m+2)$-combs ($C$).} 
\label{f:mm+1-digraphs}
\end{figure}

In the digraphs shown in Fig.~\ref{f:mm+1-digraphs}, each walk
starting at the $\bar{0}$~node and ending at an exit node corresponds
to a configuration of interior combs such that the corresponding
metatile starts with $Cc$ (thereby leaving a right slot empty).
Notice that starting from either exit node it is not possible to place
a $C$ immediately followed by a $c$ since this would result in the end
of the metatile and the $Cc$ would therefore not be interior combs.
In the present digraph, the zero-length walk starting at $\bar{0}$
corresponds to the simplest mixed metatile (which has no interior
tiles). This differs from the interior tile digraph for tiling with
$h$ and $C$ since the simplest mixed metatiles in that case do have
interior tiles.  In the same sense as for the comb digraph in the
previous section, the nodes can be grouped into the pairs
$(\bar{0}1)^p$ and $(01)^{p+1}$ for $p=0,\ldots,m$.

To show the bijection in the following lemma, in
Fig.~\ref{f:fFf-digraphs} we present the digraphs for tiling an
$n$-board using $F_m$ ($F$), $F_{m-1}$ ($f$), and $\bar{F}_1$
($\bar{f}$).  Starting from a $0^210^{2p}1$~node for any
$p=0,\ldots,m-2-j$, we can place the left side of an $F_{m-j}$ where
$j=0,1$ in the first slot. It is straightforward to show that we are
then forced to add $p+2$ $\bar{f}$, which takes us to the
$0^210^{2(m-p-4+j)}1$ node, after which we then have a choice of what
tile to place next.

\begin{figure}
\begin{center}
\includegraphics[width=12.7cm]{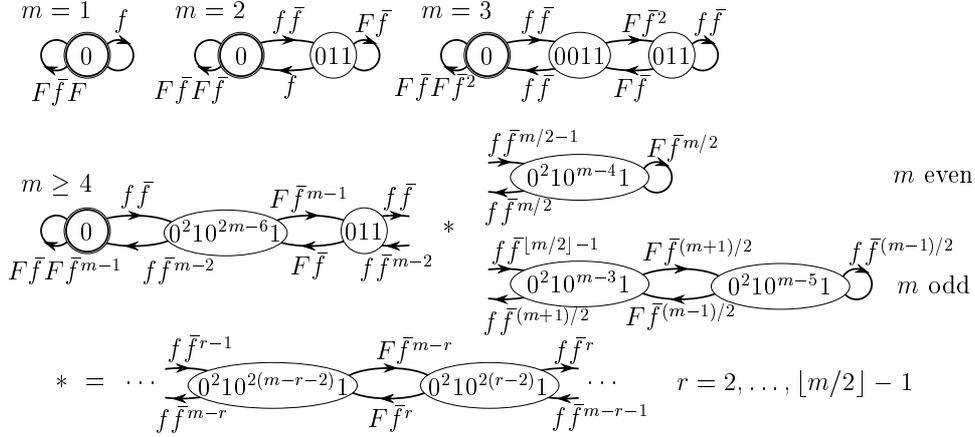}
\end{center}
\caption{Digraphs for tiling using $F_m$ ($F$),
  $F_{m-1}$ ($f$), and $\bar{F}_1$ ($\bar{f}$).}
\label{f:fFf-digraphs}
\end{figure}

\begin{lemma}\label{L:bijII}
There is a bijection between the possible pairs of configurations of
interior tiles of total length $n$ (not counting the gap when leaving
the digraph at the 01~node) when tiling
with $(\frac12,\frac12;m+1)$- and $(\frac12,\frac12;m+2)$-combs and the ways to
tile an $n$-board using $F$, $f$, and $\bar{f}$.
\end{lemma}
\begin{proof}
The proof is similar to that of Lemma~\ref{L:bijI}. Each walk starting
and ending at the 0 node corresponds to a walk starting at the
$\bar{0}$ node and ending at either exit node. Arcs containing $f$
($F$) correspond to $c$ ($C$) arcs.  The 011 node corresponds to the
$\bar{0}1$ pair. For $m\ge3$ and $r=0,\ldots,m-3$, the
$0^210^{2(m-r-3)}$ node corresponds to the $(01)^{m-r}$ pair.  We now
show that the corresponding tilings are of the same length. For any
two fence digraph nodes which are linked to each other by an arc each
way it can be seen that both arcs either contain an $f$ or an $F$ and
the number of tiles in both arcs total $m+1$ and $m+2$, respectively,
which are precisely the total lengths of two $c$ and two $C$,
respectively, that the arcs correspond to. If one of the arcs is
traversed on the outward journey, the other must be traversed on the
return journey to the 0~node. This leaves the loop from the 0~node to
itself containing two $F$ (which corresponds to two interior $C$) and
$m+2$ tiles in total, and the loop at the other end of the fence
digraph which contains an $F$ ($f$) and $(m+2)/2$ ($(m+1)/2$) tiles in
total if $m$ is even (odd).
\end{proof}

\begin{theorem}\label{T:mu2}
For $m\ge0$ and $l\ge2m+3$,
\begin{equation}\label{e:mu2}
\mu_l^{(m+1,m+2)}=2P_{l-2m-3}^{\{-2,m-1,m\}}.
\end{equation}
\end{theorem}
\begin{proof}
As there is a bijection between permutations $\pi$ of $\Nset_n$
satisfying $\pi(i)-i\in\{-2,m-1,m\}$ and tilings of an $n$-board with
$\bar{f}$, $f$, and $F$, it follows that $P_{l-2m-3}^{\{-2,-1,m\}}$ is
the number of ways to tile an $(l-2m-3)$-board with $F$, $\bar{f}$,
and $\bar{S}$. By Lemma~\ref{L:bijII} and the fact that the total
length of the interior tiles is $2m+3$ less than the complete
metatile, $P_{l-2m-3}^{\{-2,m-1,m\}}$ equals the number of pairs of
mixed metatiles of length $l$.
\end{proof}

Note that $\mu^{(1,4)}_{n+1}=\mu^{(3,4)}_{n+7}=P^{\{-2,-1,2\}}_n$ (which
is sequence A080013 in \cite{Slo-OEIS} and is given by
$P_n=P_{n-2}+P_{n-3}+P_{n-4}-P_{n-6}+\delta_{n,0}-\delta_{n,2}$,
$P_{n<0}=0$) and so the squares of $s_n^{(1,4)}$ and $s_n^{(3,4)}$ are
both closely related to this sequence.

\begin{identity}\label{I:gen2}
For $n\ge0$,
\begin{equation}\label{e:gen2}
(s_n^{(m+1,m+2)})^2\!\!=\!\delta_{n,0}+(s_{n-m-1}^{(m+1,m+2)})^2
\!+(s_{n-m-2}^{(m+1,m+2)})^2\!+2\!\!\!\!\!\sum_{l=2m+3}^n\!\!\!\!\!\!P_{l-2m-3}^{\{-2,m-1,m\}}(s_{n-l}^{(m+1,m+2)})^2.
\end{equation}
\end{identity}
\begin{proof}
The proof is analogous to that for Identity~\ref{I:gen1}.
\end{proof}

The number of ways to tile an $n$-board using only $c^2$ and $C^2$ is
$s_n^{(m+1,m+2)}$ since this is equivalent to tiling an $n$-board with
$(m+1)$- and $(m+2)$-ominoes \cite{BCCG96,BQ=03}.  Then the proof of
the following identity mirrors that of Identity~\ref{I:s2-sI}.

\begin{identity}\label{I:s2-sII}
For $n\ge0$,
\begin{equation}\label{e:s2-sII}
(s_{n}^{(m+1,m+2)})^2=s_{n}^{(m+1,m+2)}+2\!\!\!\sum_{k=0}^{n-2m-3}\!\!\sum_{r=2m+3}^{n-k}
\!\!\!P_{r-2m-3}^{\{-2,m-1,m\}}s_{k}^{(m+1,m+2)}(s_{n-k-r}^{(m+1,m+2)})^2.
\end{equation}
\end{identity}

\section{Identities relating the Narayana's cows and Padovan numbers}
\label{s:cows}

The sequences $c_n\equiv s_n^{(1,3)}$ and $p_n\equiv s_n^{(2,3)}$ are
known, respectively, as the Narayana's cows and Padovan numbers. From
Theorem~\ref{T:-1}, $P_n^{\{-2,-1,1\}}=P_n^{\{2,1,-1\}}=p_n$
and $P_n^{\{-2,0,1\}}=P_n^{\{-1,0,2\}}=c_n$. This means that putting $m=1$ in
Identities \ref{I:gen1}, \ref{I:gen2}, \ref{I:n+m+2},
\ref{I:n(m+2)+jI}, \ref{I:s2-sI}, and \ref{I:s2-sII} gives,
respectively, the following identities relating $c_n$ and $p_n$.

\begin{identity}\label{I:genc}
For $n\ge0$,
\[
c_n^2=\delta_{n,0}+c_{n-1}^2
+c_{n-3}^2+2\sum_{l=3}^np_{l-1}c_{n-l}^2.
\]
\end{identity}

\begin{identity}\label{I:genp}
For $n\ge0$,
\[
p_n^2=\delta_{n,0}+p_{n-2}^2
+p_{n-3}^2+2\sum_{l=5}^nc_{l-5}p_{n-l}^2.
\]
\end{identity}

\begin{identity}\label{I:cn+3}
For $n\ge0$,
\[
c_{n+3}^2-1=\sum_{k=0}^n\left\{c_k^2+2\sum_{i=0}^kp_{k+2-i}c_i^2\right\}.
\]
\end{identity}

\begin{identity}\label{I:c3n+j}
For $n\ge0$ and $j=0,1,2$,
\[
c_{3n+j}^2=\delta_{j,0}+(1-\delta_{j,0})c_{j-1}^2+\sum_{k=1}^n\left\{c_{3k+j-1}^2
+2\!\!\!\!\!\!\sum_{i=0}^{3(k-1)+j}\!\!\!\!\!\!p_{3k+j-1-i}c_i^2\right\}.
\]
\end{identity}

\begin{identity}\label{I:c2-c}
For $n\ge0$,
\[
c_n^2=c_n+2\sum_{k=0}^{n-3}\sum_{r=3}^{n-k}
p_{r-1}c_{k}c_{n-k-r}^2.
\]
\end{identity}

\begin{identity}\label{I:p2-p}
For $n\ge0$,
\[
p_n^2=p_n+2\sum_{k=0}^{n-5}\sum_{r=5}^{n-k}
c_{r-5}p_{k}p_{n-k-r}^2.
\]
\end{identity}

\section{Discussion}
Although various methods have been devised to find expressions for
permanents of (0,1) Toeplitz matrices (or the equivalent problem of
enumerating the number of strongly restricted permutations)
\cite{CCR97,Sta=11,Klo09,Bal10,KOI19}, our results concerning
$P_n^{\{-2,M-1,m\}}$ appear from the literature to be only the second
explicit connection, after that of certain cases of
Theorem~\ref{T:-1}, made between permanents of (0,1) Toeplitz matrices
with three nonzero diagonals none of which is the leading diagonal and
other sequences. We have further results on $P_n^W$ obtained using
fences, but not combs, that we will present elsewhere.  As for combs,
techniques used previously to obtain bijective proofs of identities
using fences can be modified for comb tilings and applied to obtain
further identities concerning other sequences in a straightforward
manner.


\end{document}